\documentclass[10pt]{article}
\usepackage{verbatim}
\usepackage{amssymb}
\usepackage{amsthm}
\usepackage{amsmath}
\usepackage[usenames,dvipsnames] {color}
\numberwithin{equation}{section}

\title{An Approximate Version of the Jordan von Neumann Theorem for Finite Dimensional Real Normed Spaces}
\author{Benjamin Passer}

\begin{document}

\bibliographystyle{plain}
\maketitle

\newtheorem{defin}[equation]{Definition}
\newtheorem{lem}[equation]{Lemma}
\newtheorem{prop}[equation]{Proposition}
\newtheorem{thm}[equation]{Theorem}
\newtheorem{claim}[equation]{Claim}
\newtheorem{ques}[equation]{Question}
\newtheorem{fact}[equation]{Fact}
\newtheorem{axiom}[equation]{Technical Axiom}
\newtheorem{newaxiom}[equation]{New Technical Axiom}
\newtheorem{cor}[equation]{Corollary}

Washington University in St. Louis    ;       E-mail: bpasser@math.wustl.edu

\begin{abstract}
It is known that any normed vector space which satisfies the parallelogram law is actually an inner product space. For finite dimensional normed vector spaces over $\mathbb{R}$, we formulate an approximate version of this theorem: if a space approximately satisfies the parallelogram law, then it has a near isometry with Euclidean space. In other words, a small von Neumann Jordan  constant $\varepsilon + 1$ for $X$ yields a small Banach-Mazur distance with $\mathbb{R}^n$, $d(X, \mathbb{R}^n) \leq 1 + \beta_n \varepsilon + O(\varepsilon^2)$. Finally, we examine how this estimate worsens as $n = dim(X)$ increases, with the conclusion that $\beta_n$ grows quadratically with $n$.
\end{abstract}



\section{Introduction}

A well known theorem of Jordan and von Neumann (see \cite{jordan}) states that every real or complex normed vector space $X$ which satisfies the paralellogram law
\begin{equation*} ||a + b||^2 + ||a - b||^2 = 2(||a||^2 + ||b||^2) \end{equation*}
for all $a, b \in X$ is an inner product space. Perhaps the most interesting facet of this theorem is its connection of a condition on arbitrary two-dimensional subspaces to a property described in at least three dimensions, namely, the additivity of the inner product:
\begin{equation*} \langle a + b, c \rangle = \langle a, c \rangle + \langle b, c \rangle. \end{equation*}
Alternatively, if each two dimensional subspace of $X$ is isometric with $\mathbb{R}^2$ (or $\mathbb{C}^2$), then $X$ is an inner product space. 

In real inner product spaces, we may describe the inner product in terms of the norm in a number of ways, such as
\begin{equation*} \langle a, b \rangle  = \cfrac{1}{4}(||a + b||^2 - ||a - b||^2), \end{equation*}
and each such expression can be evaluated even in spaces equipped only with a norm. This, then, leads to the following observation: approximate satisfaction of the parallelogram law leads to approximate bilinearity of the operator $\langle \cdot , \cdot \rangle$ as defined above. Approximate satisfaction of the paralleogram law is measured by the von Neumann-Jordan constant, 
the smallest $M \in [1, 2]$ such that 
\begin{equation} \label{eq:main}\cfrac{1}{M} \le \cfrac{ ||a + b||^2 + ||a - b||^2 }{ 2(||a||^2 + ||b||^2) } \le M \end{equation}
for all nonzero $a, b \in X$. 
Using this idea, we will show that the Jordan von Neumann theorem has an approximate version for finite dimensional spaces. That is, if $X$ has a small von Neumann-Jordan constant, then the Banach-Mazur distance between $X$ and $\mathbb{R}^n = \ell_2^n$ is close to $1$. Recall that for two
normed vector spaces $X$ and $Y$ over the same field with the same finite dimension, the Banach-Mazur distance between $X$ and $Y$ is 
\begin{equation*} d(X, Y) = \textrm{inf}\{ || \Lambda ||\cdot| |\Lambda ||^{-1}: \Lambda : X \to Y \textrm{ is a linear isomorphism}\} \geq 1. \end{equation*} 

This result is very simple in nature, and to my knowledge (and surprise) it appears not to have been previously discussed. 
Of course, the von Neumann-Jordan constant remains of other interest. For example, recent work by F. Wang has provided the final push to show that the von Neumann-Jordan constant $C_{NJ}(X)$ satisfies
\begin{equation*} C_{NJ}(X) \leq \textrm{sup} \{ \textrm{min} \{ ||x + y||, ||x - y|| \}: ||x|| = ||y|| = 1\}, \end{equation*}
the latter expression being known as the James constant (see \cite{wang}). Slightly closer in flavor to the results presented here, a result of K. Hashimoto and G. Nakamura in \cite{approx} establishes that a different type of approximate Jordan von Neumann Theorem fails. They demonstrate that the modified von Neumann-Jordan constant
\begin{equation*} \tilde{C}_{NJ}(X) = \textrm{inf} \{ C_{NJ}(X, | \cdot |): |\cdot| \textrm{ is an equivalent norm to } ||\cdot|| \textrm{ on } X\}\end{equation*}
of value 1 is not sufficient to determine if a Banach space $X$ is isomorphic with a Hilbert space. Of course, this problem deliberately blurs the distinction between isometry and isomorphism, and as such it is not relevant for finite-dimensional spaces. Regardless, we should expect that any relationship we establish tying $C_{NJ}(X)$ to $d(X, \mathbb{R}^n)$ must grow worse as $n = \textrm{dim}(X)$ increases,   a necessity that can also be seen by examining the Lebesgue spaces with increasing finite dimension. With all this in mind, we arrive at the precise statement of the main theorem (later proved as Theorem \ref{thm:main2}).
\begin{thm}
	Let $n \geq 2$. Then there is a function $K_n(\varepsilon) = 1 + (18n^2 - 17n + 14)\varepsilon + O(\varepsilon^2)$ such that for any $n$-dimensional real normed vector space $X$ with von Neumann-Jordan constant $\varepsilon + 1$ sufficiently small, $d(X, \mathbb{R}^n) \leq K_n(\varepsilon)$.
\end{thm}
The 2-dimensional case is proved as Theorem \ref{thm:main} with estimates very similar to the equalities used in \cite{jordan}. The main theorem then results from an inductive argument, in which we use an idea from algebraic topology to establish quadratic growth of the the linear term of $K_n(\varepsilon)$ as $n$ increases.

\section{Two Dimensions}
For the rest of this section, $X$ is a real normed vector space with von Neumann-Jordan constant $M$, so
\begin{equation} \cfrac{1}{M} \le \cfrac{ ||a + b||^2 + ||a - b||^2 }{ 2(||a||^2 + ||b||^2) } \le M \end{equation}
for all nonzero $a, b \in X$. The inequality above will be most useful through its consequential inequality
\begin{equation} \label{eq:2} \left|||a + b||^2 + ||a - b||^2 - 2(||a||^2 + ||b||^2)\right| \le 2\varepsilon(||a||^2 + ||b||^2), \end{equation}
where $\varepsilon = M - 1$ ($\varepsilon$ is greater than $1 - \frac{1}{M}$, but comparable as $M$ approaches 1). 
We use these inequalities to establish that the \textit{bracket}
\begin{equation*} [a, b] = ||a + b||^2 - ||a - b||^2 \end{equation*}
 is approximately bilinear ($[a, b] = 4\langle a , b \rangle$ if $X$ is a Hilbert space). As an easy first step, we note that the bracket is symmetric and homogeneous of degree 2 on $X^2$: 
\begin{equation} \label{eq:homo} [\lambda a, \lambda b] = \lambda^2 [a, b]. \end{equation}
Moreover, homogeneity in one coordinate with scalar $-1$ is trivial:
\begin{equation} \label{eq:negatives} [-a, b] = -[a, b]. \end{equation}
The final immediate observation on the bracket is the Cauchy-Schwarz inequality, scaled by four due to the choice of bracket. (Namely, the bracket is four times the inner product in an inner product space.)
\begin{lem}\label{csi}
Let $x, y \in X$. Then
\begin{equation*}
	|[x, y]| \leq 4||x|| \cdot ||y||.
\end{equation*}
\end{lem}
\begin{proof}
The triangle inequality gives that $||x + y||^2 \leq (||x|| + ||y||)^2 = ||x||^2 + ||y||^2 + 2||x|| \cdot ||y||$. Moreover, $||x - y||^2 \geq \big| ||x|| - ||y|| \big|^2 = ||x||^2 + ||y||^2  - 2||x|| \cdot ||y||$. Subtracting these two inequalities yields 
\begin{equation*}
	[x, y] = ||x + y||^2 - ||x - y||^2 \leq 4||x|| \cdot ||y||.
\end{equation*}
Similarly, $||x + y||^2 \geq \big| ||x|| - ||y|| \big|^2 = ||x||^2 + ||y||^2 - 2||x|| \cdot ||y||$ and $||x - y||^2 \leq (||x|| + ||y||)^2 = ||x||^2 + ||y||^2 + 2||x|| \cdot ||y||$ give 
\begin{equation*}
	[x, y] = ||x + y||^2 - ||x - y||^2 \geq -4||x|| \cdot ||y||.
\end{equation*}
\end{proof}
Armed only with the above observations and the von Neumann-Jordan constant, we may prove that the bracket is approximately bilinear, in a series of steps analagous to those in \cite{jordan}.
\begin{lem}\label{lem:doublelaw} 
Let $x, y, z \in X$. Then 
\begin{equation*}
\begin{aligned}
	\big|[2x, y] - 2[x, y]\big| &\leq 2 \varepsilon(||x + y||^2 + ||x - y||^2 + 2||x||^2) \\
				         &\leq 4 \varepsilon \big[(1 + \varepsilon)(||x||^2 + ||y||^2) + ||x||^2\big].
\end{aligned}
\end{equation*}
\end{lem}
\begin{proof}
Apply (\ref{eq:2}) twice to see that
\begin{equation*} ||2x \pm y||^2 + ||\pm y||^2 \approx 2(||x \pm y||^2 + ||x||^2), \end{equation*}
with error at most $\varepsilon$ times the right hand side. Subtract the two approximate equalities, resulting in $|[2x, y] - 2[x, y]| \leq 2\varepsilon(||x + y||^2 + ||x - y||^2 + 2 ||x||^2)$. We may then apply
(\ref{eq:2}) again to the error term:
\begin{equation*}
     	\begin{aligned} 2 \varepsilon(||x + y||^2 + ||x - y||^2 + 2||x||^2) &\leq 2 \varepsilon \big[2(1 + \varepsilon)(||x||^2 + ||y||^2) + 2||x||^2\big] \\
                                                                                                             &= 4 \varepsilon \big[(1 + \varepsilon)(||x||^2 + ||y||^2) + ||x||^2 \big]. 
	\end{aligned}
\end{equation*}
\end{proof}
\begin{lem}\label{lem:addlaw} Let $x, y, z \in X$. Then 
\begin{equation*}	
	\big|[x, z] + [y, z] - [x + y, z]\big| \leq \varepsilon \big[(3 + 2 \varepsilon) ||x + y||^2 + ||x - y||^2 + 8(1 + \varepsilon) ||z||^2 \big].	
\end{equation*}
\end{lem}
\begin{proof} 
Apply (\ref{eq:2}) twice to see that 
\begin{equation*} 
	||x \pm z||^2 + ||y \pm z||^2 \approx 2(||\cfrac{x + y}{2} \pm z||^2 + ||\cfrac{x - y}{2}||^2), 
\end{equation*}
with error at most $\varepsilon$ times the right hand side.
Subtracting the above two expressions gives 
\begin{equation*} 
	\left| [x, z] + [y, z] - 2[\cfrac{x + y}{2}, z] \right| \leq 2 \varepsilon\left(||\cfrac{x + y}{2} + z||^2 + ||\cfrac{x + y}{2} - z||^2 + 2 || \cfrac{x - y}{2} ||^2\right)
\end{equation*}
We then apply the first inequality of Lemma \ref{lem:doublelaw}, which implies that 
\begin{equation*}
\begin{aligned}
	\big| [x + y, z] - 2[\cfrac{x + y}{2}, z] \big| &\leq 2 \varepsilon \left(||\cfrac{x + y}{2} + z||^2 + ||\cfrac{x + y}{2} - z||^2 + 2 ||\cfrac{x + y}{2}||^2 \right).
\end{aligned}
\end{equation*}
Thus, the total error in the approximation of $[x + y, z]$ by $[x, z] + [y, z]$ is at most
\begin{equation*}
	2 \varepsilon\left(2 ||\cfrac{x + y}{2} + z||^2 + 2 ||\cfrac{x + y}{2} - z||^2 + 2 || \cfrac{x + y}{2} ||^2 + 2||\cfrac{x - y}{2}||^2 \right) \leq
\end{equation*}
\begin{equation*}
	2 \varepsilon \left[ (1 + \varepsilon)(||x + y||^2 + ||2z||^2) + \cfrac{1}{2}||x + y||^2 + \cfrac{1}{2}||x - y||^2 \right] =
\end{equation*}
\begin{equation*} 
	\varepsilon \left[ (3 + 2\varepsilon)||x + y||^2 + ||x - y||^2  + 8(1 + \varepsilon)||z||^2 \right] \\
\end{equation*}
\end{proof}
\begin{lem}\label{lem:scalelaw} Let $x, y \in X$ and $t \in \mathbb{R}$. Then 
\begin{equation*}
	\big| [tx, y] - t[x, y] \big| \leq \textrm{max}\{1, t^2\}\varepsilon \big[ (8 + 7\varepsilon)||x||^2 + (20 + 20\varepsilon)||y||^2 \big].
\end{equation*}
\end{lem}
\begin{proof} Given fixed $x, y \in X$ and $n \in \mathbb{Z}_{\geq 0}$, let
\begin{equation*} S_n = \textrm{max}_{0 \leq k \leq 2^n} \left|[\frac{k}{2^n}, y] - \frac{k}{2^n}[x, y]\right|. \end{equation*}
Clearly $S_0 = 0$, and Lemma $\ref{lem:doublelaw}$ (scaled by half) implies
\begin{equation*} 
\begin{aligned}
	S_1 &\leq 2 \varepsilon \big[(1 + \varepsilon)(||\frac{1}{2}x||^2 + ||y||^2) + ||\frac{1}{2}x||^2\big] \\
	       &= 2 \varepsilon \big[(1 + \varepsilon)(\frac{1}{4}||x||^2 + ||y||^2) + \frac{1}{4}||x||^2\big].
\end{aligned}
\end{equation*}
We can then bound $S_{n + 1}$ in terms of $S_n$.
First, suppose $\cfrac{k}{2^{n + 1}} \in [0, \frac{1}{2}]$. Applying Lemma $\ref{lem:doublelaw}$ in the same way as above, we get that
\begin{equation*}
\begin{aligned}
	| [\frac{k}{2^{n + 1}}x, y] - \frac{1}{2} [\frac{k}{2^n}x, y]| &\leq 2 \varepsilon ((1 + \varepsilon)(||\frac{k}{2^{n + 1}}x||^2 + ||y||^2) + ||\frac{k}{2^{n + 1}}x||^2) \\
										      &\leq 2 \varepsilon ((1 + \varepsilon)(\frac{1}{4}||x||^2 + ||y||^2) + \frac{1}{4}||x||^2).
\end{aligned}
\end{equation*}
Denote the last value by $A$. We then know that $|[\frac{k}{2^{n + 1}}x, y] - \frac{k}{2^{n + 1}}[x, y]| \leq \frac{S_n}{2} + A$, since $\frac{k}{2^n} \in [0, 1]$. 
Next, we apply Lemma $\ref{lem:addlaw}$ to see that
\begin{equation*}
\begin{aligned}
	\left| [\frac{k}{2^{n+1}}x, y] + [(1 - \frac{k}{2^{n+1}})x, y] - [x, y] \right| &\leq \varepsilon \big[ (3 + 2\varepsilon) ||x||^2 + ||(1 - 2 \frac{k}{2^{n+1}})x||^2 + 8(1 + \varepsilon)||y||^2 \big] \\
													&\leq \varepsilon \big[ (3 + 3\varepsilon) ||x||^2 + 8 (1 + \varepsilon) ||y||^2 \big]. \\
\end{aligned}
\end{equation*}
If we denote the last value by $B$, then we have that
\begin{equation*} S_{n + 1} \leq \frac{S_n}{2} + A + B = \frac{S_n}{2} + \varepsilon \left[ (4 + \frac{7}{2}\varepsilon)||x||^2 + (10 + 10\varepsilon)||y||^2 \right].   \end{equation*}

This relationship allows us to induct and show $S_n \leq 2(A + B)$ for all $n$. But, by continuity of the bracket, $|[x, ty] - t[x, y]| \leq 2(A + B)$ for all $t \in [0, 1]$. This then proves the lemma for $t \in [0, 1]$. To obtain the desired result for all $t > 1$, we apply (\ref{eq:homo}) and switch the roles of $x$ and $y$ (also noting that the bracket is symmetric) to see that
\begin{equation*}
	\big| [tx, y] - t[x, y] \big| = t^2 \left| [x, \frac{1}{t}y] - \frac{1}{t}[x, y] \right| \leq t^2 \cdot 2(A + B).
\end{equation*}
Finally, for negative $t$, we simply note $[-a, b] = -[a, b]$ by definition of the bracket, and the entire lemma follows.
\end{proof}
\begin{thm}\label{thm:main} Let $X$ be a two-dimensional normed vector space over $\mathbb{R}$ with von Neumann-Jordan constant $\varepsilon + 1$. Then the Banach-Mazur distance between $X$ and $\mathbb{R}^2$ is at most $\sqrt{\cfrac{1 + 15\varepsilon + 13.5 \varepsilon^2}{1 - 15 \varepsilon - 13.5 \varepsilon^2}}$ provided $\varepsilon$ is small enough that the denominator is positive. This expression
is $1 + 15\varepsilon + O(\varepsilon^2)$ as $\varepsilon$ tends to $0$. \end{thm}
\begin{proof}
Let $x$ be an arbitrary unit element in $X$. Since $[x, x] = -[x, -x] = 4$, connectedness of the unit sphere in $X$ implies there is a unit vector $y \in X$ with $[x, y]$ = 0. Let $\Lambda: X \to \mathbb{R}^2$ send $ax + by$ to $(a, b)$. For $t \in \mathbb{R}$, (\ref{eq:main}) implies that
\begin{equation*} \left| \cfrac{ ||x + ty||^2 + ||x - ty||^2}{2(1 + t^2)} - 1 \right| \leq \varepsilon. \end{equation*}
Because $[x, y] = 0$, Lemma \ref{lem:scalelaw} shows that 
\begin{equation*}
\begin{aligned}
	\big| ||x + ty||^2 - ||x - ty||^2 \big| &= \big| [x, ty] | \\
						      &\leq \textrm{max}\{1, t^2\}\varepsilon \big[ (8 + 7\varepsilon)||x||^2 + (20 + 20\varepsilon)||y||^2 \big] \\
						      &\leq (1 + t^2)\varepsilon(28 + 27\varepsilon).
\end{aligned}
\end{equation*}
The combination of these two inequalities implies that 
\begin{equation*} \left| \cfrac{||x + ty||^2}{1 + t^2} - 1\right| \leq \varepsilon + \varepsilon(14 + 13.5 \varepsilon) =  15\varepsilon + 13.5 \varepsilon^2. \end{equation*}
That is, 
\begin{equation*} 1 - 15\varepsilon - 13.5 \varepsilon^2 \leq \cfrac{||x + ty||^2}{||\Lambda(x + ty)||^2} \leq 1 + 15 \varepsilon + 13.5 \varepsilon^2 \end{equation*}
The set $\{x + ty: t \in \mathbb{R}\}$ is a sufficient testing set for calculating $||\Lambda||$ and $||\Lambda||^{-1}$, so the desired result follows.
\end{proof}

\section{Higher-Dimensional Results}
To extend the result of Theorem \ref{thm:main} to higher dimensions, we apply an induction argument. The following lemma shows the approximate bilinearity of the bracket in this setting.
\begin{lem}\label{lem:linear} Suppose $X$ is a normed, real vector space with von Neumann-Jordan constant $M = \varepsilon + 1$. Also, let $x_1, ... , x_k$ be the images of standard basis vectors under a linear map $\Lambda: \mathbb{R}^k \to X$ with $||\Lambda|| \leq 1$. 
Then, for any $y \in X$ with $||y|| \leq 1$ and for any $a_1, . . . , a_{n - 1} \in \mathbb{R}$,
\begin{equation*} \cfrac{\left| [\sum_{i = 1}^k a_ix_i, y] - \sum_{i = 1}^k a_i [x_i, y] \right|}{1 + \sum_{i = 1}^k a_i^2} \leq 36 k \varepsilon + O(\varepsilon^2). \end{equation*} 
Note that the $O(\varepsilon^2)$ terms do depend on $k$, but not on $y$ or any $a_i$.
\end{lem}
\begin{proof}
First, we establish some terminology. A subsum of $\sum_{i = 1}^k a_i x_i$ is a formal sum $A = \sum_{j = 1}^m a_{i_j} x_{i_j}$ written in reduced form. Two subsums are called disjoint if they contain no common $x_i$ terms. 
The fact that $||\Lambda|| \leq 1$ implies that $||x_i|| \leq 1$ and more generally $||A||^2 \leq \sum_{j = 1}^m {a_{i_j}}^2$. Consequently, if $A$ and $B$ are disjoint subsums, then $||A \pm B||^2 \leq \sum_{i = 1}^k a_i^2$.
Now, if $k = 1$, then Lemma \ref{lem:scalelaw} implies that
\begin{equation*}
\begin{aligned}
	|[a_1 x_1, y] - a_1[x_1, y]| &\leq \textrm{max}\{1, a_1^2\}\varepsilon\big[ (8 + 7 \varepsilon)||x_1||^2 + (20 + 20\varepsilon)||y||^2 \big] \\
					          &\leq \textrm{max}\{1, a_1^2\}\varepsilon\big[28 + 27\varepsilon\big] \\
					          &\leq (1 + a_1^2)\varepsilon\big[28 + 27\varepsilon\big]. \\
\end{aligned}
\end{equation*}
Now, let $k$ be arbitrary, and let $A$ and $B$ be disjoint subsums of $\sum_{i = 1}^k a_ix_i$ with $A + B = \sum_{i = 1}^k a_ix_i$. So, $||A \pm B||^2 \leq \sum_{i = 1}^k a_i^2$ and by Lemma \ref{lem:addlaw}, 
\begin{equation*}
\begin{aligned}
	\big| [A + B, y] - [A, y] - [B, y] \big| &\leq \varepsilon \big[ (3 + 2\varepsilon)||A + B||^2 + ||A - B||^2 + 8(1 + \varepsilon)||y||^2\big] \\
										 &\leq \varepsilon\big[(4 + 2\varepsilon)\sum_{i = 1}^k a_i^2 + 8 + 8\varepsilon\big].
\end{aligned}
\end{equation*}
This sets up an induction on $k$, immediately yielding that $\left| [\sum_{i = 1}^k a_ix_i, y] - \sum_{i = 1}^k a_i [x_i, y] \right|$ is bounded above by
\begin{equation*}
\begin{aligned}
	(k - 1) \varepsilon\big[(4 + 2\varepsilon)\sum_{i = 1}^k a_i^2 + 8 + 8\varepsilon\big] + (k + \sum_{i = 1}^k a_i^2)\varepsilon\big[28 + 27\varepsilon\big].
\end{aligned}
\end{equation*}
This expression is bounded above by an expression of the form
\begin{equation*}
	(k - 1)\varepsilon (4 \sum_{i = 1}^k a_i^2 + 8) + \varepsilon(28k + 28 \sum_{i = 1}^ka_i^2) + (1 + \sum_{i = 1}^k a_i^2) O(\varepsilon^2),
\end{equation*}
which is in turn bounded by 
\begin{equation*}
	36 k \varepsilon(1 + \sum_{i = 1}^k a_i^2) + (1 + \sum_{i = 1}^k a_i^2) O(\varepsilon^2).
\end{equation*}
\end{proof}
Now, the approximate bilinearity of the bracket established above will help show $X$ is nearly isometric with a direct sum of two of its smaller dimensional subspaces. The norm in the direct sum is defined in a Euclidean fashion.
\begin{lem}\label{lem:induct}
	Suppose $X$ is an $n$-dimensional real vector space with von Neumann-Jordan constant $\varepsilon + 1$ and $\Lambda: \mathbb{R}^{n - 1} \to Y \leq X$ is an invertible map with $||\Lambda|| \leq 1$, $||\Lambda^{-1}|| \leq K$ (so $K \geq 1)$. If $x_1, \ldots, x_{n - 1} \in X$ are the images of the standard basis 
vectors of $\mathbb{R}^{n - 1}$ under $\Lambda$ and $x_n \in X$ has $||x_n|| = 1$ and $|[x_i, x_n]| \leq \delta$ for $i < n$, then
\begin{equation*}
	\left| \cfrac{||\sum_{i = 1}^{n - 1} a_i x_i + x_n||^2}{||\sum_{i = 1}^{n - 1} a_i x_i||^2 + 1} - 1 \right| \leq \varepsilon + \cfrac{K^2}{2}(\sqrt{n - 1}\delta + 36(n - 1)\varepsilon + O(\varepsilon^2))
\end{equation*}
for any $a_1, \ldots, a_{n - 1} \in \mathbb{R}$.
\end{lem}
\begin{proof}
	The approximate parallegram law ($\ref{eq:main}$) and the fact that $||\Lambda^{-1}|| \leq K$ imply that
\begin{equation*}
\left| \cfrac{||\sum_{i = 1}^{n - 1}a_i x_i + x_n||^2}{||\sum_{i = 1}^{n - 1}a_i x_i||^2 + 1} - 1\right| \leq 
\end{equation*}
\begin{equation*}
	\left | \cfrac{ ||\sum_{i = 1}^{n - 1}a_i x_i + x_n||^2 + ||\sum_{i = 1}^{n - 1}a_i x_i - x_n||^2}{2( ||\sum_{i = 1}^{n - 1}a_i x_i ||^2 + 1)}  - 1 \right| + \cfrac{1}{2} \cfrac{ | [\sum_{i = 1}^{n - 1}a_i x_i, x_n] | }{ ||\sum_{i = 1}^{n - 1}a_i x_i||^2 + 1} \leq 
\end{equation*}
\begin{equation*}	
	\varepsilon +  \cfrac{K^2}{2} \cfrac{ | [\sum_{i = 1}^{n - 1}a_i x_i, x_n] | }{ \sum_{i = 1}^{n - 1}a_i^2 + 1}.
\end{equation*}
Now, we may apply Lemma $\ref{lem:linear}$ to see that this is bounded above by 
\begin{equation*}
	\varepsilon + \cfrac{K^2}{2} \cdot \left( \cfrac{|\sum_{i = 1}^{n - 1}a_i[x_i, x_n]|}{\sum_{i = 1}^{n - 1}a_i^2 + 1} + 36(n - 1)\varepsilon + O(\varepsilon^2) \right) \leq
\end{equation*}
\begin{equation*}
	\varepsilon + \cfrac{K^2}{2} \cdot \left( \cfrac{\delta \sum_{i = 1}^{n - 1}|a_i| }{\sum_{i = 1}^{n - 1}a_i^2 + 1} + 36(n - 1)\varepsilon + O(\varepsilon^2)\right) \leq 
\end{equation*}
\begin{equation*}
	\varepsilon + \cfrac{K^2}{2} \cdot \left( \sqrt{n - 1}\delta + 36(n - 1)\varepsilon + O(\varepsilon^2)\right).
\end{equation*}
\end{proof}
If $\varepsilon$ and $\delta$ are small enough, the lemma proves a bound on $d(X, Y \oplus \mathbb{R})$. 
Since our overall goal is to establish that a small von Neumann-Jordan constant results in a small Banach-Mazur distance between $X$ and $\mathbb{R}^{n}$, we will succeed so long as we can control the bracket of some unit vector $x_n$ with generators of an $(n - 1)$-dimensional subspace $Y$ of $X$. We would like to choose $x_n$ in such a way that the linear coefficient of $K_n(\varepsilon)$ in the expression $d(X, \mathbb{R}^n) \leq K_n(\varepsilon)$ grows slowly with $n = dim(X)$. Using a Gram-Schmidt type method leads to a linear term in $K_n(\varepsilon)$ with faster than factorial-squared growth, but we can do much better by appealing to a theorem from algebraic topology.
\begin{lem}\label{lem:ortho}
	Suppose $X$ is an $n$-dimensional real normed vector space and let $x_1, \ldots, x_{n - 1} \in X$. Then there is some $x_n \in X$ with $||x_n|| = 1$ and $[x_n, x_i] = 0$ for $i < n$.
\end{lem}
\begin{proof}
	The result is trivially true for $n = 1$, and the result for $n = 2$ follows by a simple continuity argument. For higher dimensional $X$, consider the unit sphere $S$ of $X$. Since $X$ is an $n$-dimensional normed space, there is a homeomorphism $g$ from the standard sphere $S^{n - 1}$ to $S$ such that $g(-x) = -g(x)$ for all $x \in S^{n - 1}$. Now, compose this map with the following map $f: S \to \mathbb{R}^{n - 1}$:
\begin{equation*}
	f: y \mapsto \left( \begin{array}{c} {[y, x_1]} \\ {[y, x_2]} \\ \vdots \\ {[y, x_{n - 1}]} \end{array} \right).
\end{equation*}
By (\ref{eq:negatives}), $f$ is an odd map, and as such $f \circ g: S^{n - 1} \to \mathbb{R}^{n - 1}$ is also an odd map. However, the Borsuk-Ulam Theorem (see \cite{hatcher}) implies that there is a point $\alpha \in S^{n - 1}$ with $f \circ g (-\alpha) = f \circ g (\alpha)$. Oddness shows that $f \circ g (\alpha) = 0$, so we set $x_n = g(\alpha)$. 
\end{proof}
With the setup completed, the main theorem is finally in sight. Theorem $\ref{thm:main}$ shows that for 2-dimensional $X$, $d(X, \mathbb{R}^2) \leq 1 + 15\varepsilon + O(\varepsilon^2)$ as $\varepsilon$ tends to zero.
\begin{thm}\label{thm:main2}
	Let $n \geq 2$. Then there is a function $K_n(\varepsilon) = 1 + (18n^2 - 17n + 14)\varepsilon + O(\varepsilon^2)$ such that for any $n$-dimensional real normed vector space $X$ with von Neumann-Jordan constant $\varepsilon + 1$ sufficiently small, $d(X, \mathbb{R}^n) \leq K_n(\varepsilon)$.
\end{thm}
\begin{proof}
	Let $\beta_k = 18k^2 - 17k + 14$ for all $k$, and note the theorem already holds for $n = 2$ by Theorem \ref{thm:main}, so we may induct. Let $Y$ be an arbitrary $(n - 1)-$dimensional subspace of $X$. Since $Y$ has von Neumann-Jordan constant at most $\varepsilon + 1$, we know that $d(Y, \mathbb{R}^{n - 1}) \leq K_{n - 1}(\varepsilon)$ if $\varepsilon$ is small. Let 
$\Lambda: \mathbb{R}^{n - 1} \to Y$ be an isomorphism with $||\Lambda|| \leq 1$ and $||\Lambda^{-1}|| \leq K_{n - 1}$, with $x_i = \Lambda(e_i)$. By Lemma \ref{lem:ortho}, there is some
$x_n \in X$ with $||x_n|| = 1$ and $[x_n, x_i] = 0$ for $i < n$. This $x_n$ satisfies the conditions of Lemma \ref{lem:induct} with $\delta = 0$, so for any $a_1, \ldots, a_{n -  1} \in \mathbb{R}$, 
\begin{equation*}
\begin{aligned}
	\left| \cfrac{||\sum_{i = 1}^{n - 1} a_i x_i + x_n||^2}{||\sum_{i = 1}^{n - 1} a_i x_i||^2 + 1} - 1 \right| &\leq \varepsilon + \cfrac{K_{n - 1}(\varepsilon)^2}{2}(36(n - 1)\varepsilon + O(\varepsilon^2)) \\
																	&= \varepsilon + (1 + \beta_{n - 1}\varepsilon + O(\varepsilon^2))(36(n - 1)\varepsilon + O(\varepsilon^2)) \\
																	&= (36n - 35)\varepsilon + O(\varepsilon^2).
\end{aligned}
\end{equation*}
As in Theorem $\ref{thm:main}$, $\{ a_1x_1 + \ldots a_{n - 1}x_{n - 1} + x_n: a_1, \ldots a_{n - 1} \in \mathbb{R}\} = \{y + x_n: y \in Y\}$ is a sufficient testing set for calculating $||T||$ and $||T^{-1}||$, where $T$ maps $(y, t) \in Y \oplus \mathbb{R}$ to $y + tx_n \in X$. The above estimate gives us that 
\begin{equation*}
	d(X, Y \oplus \mathbb{R}) \leq 1 + 36n - 35\varepsilon + O(\varepsilon^2). 
\end{equation*} 
Now, finally, 
\begin{equation*}
\begin{aligned}
	d(X, \mathbb{R}^n) 	&\leq d(X, Y \oplus \mathbb{R}) \cdot d(Y \oplus \mathbb{R}, \mathbb{R}^n) \\
			&\leq d(X, Y \oplus \mathbb{R})\cdot  d(Y, \mathbb{R}^{n - 1}) \\
			&\leq (1 + (36n - 35)\varepsilon + O(\varepsilon^2))(1 + \beta_{n - 1}\varepsilon + O(\varepsilon^2)) \\
			&= 1 + (\beta_{n - 1} + 36n - 35)\varepsilon + O(\varepsilon^2) \\
			&= 1 + \beta_n \varepsilon + O(\varepsilon^2).
\end{aligned}
\end{equation*}
\end{proof}
The linear term of $K_n(\varepsilon)$ has growth which is, at worst, quadratic. Unfortunately, we do not know if this is an optimal growth rate, and in fact, we conjecture that it is not.
\begin{prop}
	Let $X$ be a Lebesgue space of finite dimension $n$ (i.e. $\ell_p^n$ for some $p \in [1, \infty])$ with von Neumann-Jordan constant $\varepsilon + 1$. Then $d(X, \mathbb{R}^n) \leq n^{log_2(\varepsilon + 1)/2} = 1 + log_4(n)\varepsilon + O(\varepsilon^2)$.
\end{prop}
\begin{proof}
	Clarkson proved in \cite{clarkson} that the von Neumann-Jordan constant of $\ell_p^n$ is $2^{|2 - p|/p}$, and it is easy to show that the identity map from $\ell_p^n$ to $\ell_2^n$ has $||Id|| \cdot ||Id||^{-1} = n^{|1/2 - 1/p|}$. The rest is trivial.
\end{proof}
Whether logarithmic growth of the linear term above is indicative of the general case or a result of exceptional symmetry present in the Lebesgue spaces has yet to be determined.

\section*{Acknowledgements}

I would like to thank three professors: John McCarthy and Ari Stern for remarks that led to examining this problem, Renato Feres for his explanations of algebraic topology, and all three for answering my questions. There were many.

\end{document}